\documentclass[12 pt]{amsart}
\usepackage{amsmath,amsfonts,amsthm, mathrsfs,tikz, pgfplots,multicol, graphicx,stmaryrd,placeins,hyperref,placeins,amssymb,caption}
\usepackage[margin=1in]{geometry}
\usetikzlibrary{decorations.markings}

\newtheorem*{theorem*}{Theorem}

\newcommand{\mbb}[1]{\mathbb{#1}}

\begin{document}

\tikzset{->-/.style={decoration={
  markings,
  mark=at position #1 with {\arrow{>}}},postaction={decorate}}}

%%%%%%%%%%%%%%%%%%%%%%%%%%%%%%%%%%%
%%%%%%%%%%%%%%%%%%%%%%%%%%%%%%%%%%%
%%%%%%%%%%%%%%%%%%%%%%%%%%%%%%%%%%%
\title[An infinite product on Teichm\"{u}ller space]{An infinite product on the Teichm\"{u}ller space of the once-punctured torus}
\author{Robert Hines}
\address{Mathematics Department\\
Grand Valley State University\\
A-2-178 Mackinac Hall\\
1 Campus Drive\\
Allendale, Michigan 49401 }
\email{hinesro@gvsu.edu}
\date{\today}

\begin{abstract}
We prove the identity
$$
\prod_{\gamma}\left(\frac{e^{l(\gamma)}+1}{e^{l(\gamma)}-1}\right)^{2h}=\exp\left(\frac{l_1+l_2+l_3}{2}\right),
$$
(or
$$
\prod_{\gamma}\left(\frac{t(\gamma)^2}{t(\gamma)^2-4}\right)^h=\frac{t_1+\sqrt{t_1^2-4}}{2}\cdot\frac{t_2+\sqrt{t_2^2-4}}{2}\cdot\frac{t_3+\sqrt{t_3^2-4}}{2}
$$
in trace coordinates), where the product is over all simple closed geodesics on the once-punctured torus, $l(\gamma)=2\operatorname{arccosh}(t(\gamma)/2)$ is the length of the geodesic, and $l_i$ ($t_i$) are the lengths (traces) of any triple of simple geodesics $\{\gamma_i\}$ intersecting at a single point.  The exponent $h=h(\gamma;\{\gamma_i\})$ is a positive integer ``height'' which increases as we move away from the chosen triple $\{\gamma_i\}$ in its orbit under $SL_2(\mbb{Z})$ (see Figure \ref{height} for the ``definition by picture'').  For comparison, a short proof of McShane's identity
$$
\sum_{\gamma}\frac{1}{1+e^{l(\gamma)}}=\frac{1}{2}=\sum_{\gamma}\frac{1-\sqrt{1-4/t(\gamma)^2}}{2}
$$
in the same spirit is given in an appendix.  Both proofs are elementary and proceed by ``integrating'' around the chosen triple $\{\gamma_i\}$ in its Teichm\"{u}ller orbit.
\end{abstract}
\maketitle
%%%%%%%%%%%%%%%%%%%%%%%%%%%%%%%%%%%
%%%%%%%%%%%%%%%%%%%%%%%%%%%%%%%%%%%
%%%%%%%%%%%%%%%%%%%%%%%%%%%%%%%%%%%
\section{Introduction}
The equation
$$
x^2+y^2+z^2=xyz
$$
has a interesting tradition in both geometry and number theory.  In geometry, it appears as a special case of the trace identity (due to Fricke)
$$
\operatorname{tr}(A)^2+\operatorname{tr}(B)^2+\operatorname{tr}(AB)^2=\operatorname{tr}a(A)\operatorname{tr}(B)\operatorname{tr}(AB)+\operatorname{tr}(ABA^{-1}B^{-1})+2, \ A, B\in SL_2
$$
when the trace of the commutator is $-2$.  In this case $G=\langle A,B\rangle\subseteq SL_2(\mbb{R})$ is a free group on two generators and the quotient $G\backslash\mbb{H}$ is a once-punctured torus (with complete hyperbolic metric).  The equation above (for $x,y,z>2$) parameterizes all such marked surfaces.  The Teichm\"{u}ller modular group for this torus is $\Gamma=SL_2(\mbb{Z})$ which acts in a somewhat complicated fashion in these coordinates (although easy to specify on generators).

In number theory, Markoff (\cite{markoff1}, \cite{markoff2}) discovered that the ``worst approximable'' irrational numbers, (or the indefinite binary quadratic forms with largest normalized minima, or the lowest lying geodesics on the modular surface $\Gamma\backslash\mbb{H}$) have a rigid combinatorial structure parameterized by positive integer solutions to the Diophantine equation
$$
X^2+Y^2+Z^2=3XYZ.
$$

A connection (first enunciated by Harvey Cohn, cf. \cite{cohn}) between the two can be made through the ``modular torus'', $\Gamma'\backslash\mbb{H}$, where $\Gamma'$ is the commutator subgroup of $SL_2(\mbb{Z})$.  This is the maximally symmetric torus, corresponding to $(x,y,z)=(3,3,3)$ (whose Teichm\"{u}ller orbit is exactly the integer solutions of the Markoff equation).  By taking generators for the commutator subgroup, e.g. 
$$
A=M_{\frac{0}{1}}=
\left(
\begin{array}{cc}
5&2\\
2&1\\
\end{array}
\right)=2+\frac{1}{2+\frac{1}{z}}, \ 
B=M_{\frac{1}{1}}=
\left(
\begin{array}{cc}
2&1\\
1&1\\
\end{array}
\right)=1+\frac{1}{1+\frac{1}{z}},
$$
one can generate the Markoff forms as the fixed-point forms $cx^2+(d-a)xy-by^2$ of the matrices inductively defined by
$$
\left(
\begin{array}{cc}
a&b\\
c&d\\
\end{array}
\right)=
M_{p/q}=M_{p'/q'}M_{p''/q''}, \ \frac{p'}{q'}<\frac{p}{q}=\frac{p'+q'}{p''+q''}<\frac{p''}{q''}, \ p''q'-p'q''=1.
$$
The traces of the triples of ``adjacent'' matrices satisfy $x^2+y^2+z^2=xyz$ (as would any collection constructed in a similar fashion from a pair of generators of the fundamental group of a once-punctured hyperbolic torus).

It is a challenge to determine which properties of the discrete Markoff spectrum are arithmetic in nature and particular to the modular torus, and which are generic for once-punctured tori.  For instance, a long-standing question of Frobenius \cite{frobenius} is whether or not the maximum coordinate of a Markoff triple $(X,Y,Z)$ uniquely determines the triple (in other words, whether or not there are repetitions in the traces of the matrices $M_{p/q}$ above).  This uniqueness property for the length spectrum (up to automorphisms) does \textit{not} hold for generic tori, cf. \cite{mcshaneparlier}.  However, the question/conjecture has extensive numerical verification and uniqueness has been proven for various subsets of the Markoff numbers (e.g. \cite{baragar}, \cite{button}).
%Perhaps this rephrasing will help illustrate:
%\begin{quote}
%\textit{Can the ideal class group of a real quadratic order contain at most one of these ``lowest-lying'' geodesics?}
%\end{quote}
%
%To see this rephrasing, note that a closed geodesic on the modular surface corresponds to the $SL_2(\mbb{Z})$-conjugacy class of a primitive (not a power of another) hyperbolic element of $SL_2(\mbb{Z})$.  The fixed points of such an element are the roots of an indefinite integral binary quadratic form.  Conversely, the automorphism group of an indefinite integral binary quadratic form is generated by $\pm1$ and a hyperbolic element of $SL_2(\mbb{Z})$.  The $SL_2(\mbb{Z})$-actions of conjugacy and linear change of variable are intertwined, one being the contragredient of the other.  Cf. \cite{sarnak}.
%$$
%S(a,b;c)+S(b,c;a)+S(c,b;a)=\frac{a^2+b^2+c^2-3abc}{12abc}
%$$
%where the homogeneous Dedekind sum is
%$$
%S(i,j;k)=\sum_{t=1}^kW(it/k)W(jt/k), \ W(x)=\left\{
%\begin{array}{cc}
%x-\lfloor x\rfloor-\frac{1}{2}&x\not\in\mbb{Z}\\
%0&x\in\mbb{Z}\\
%\end{array}
%\right..
%$$

The Markoff equation appears in various contexts; e.g. Rademacher reciprocity for Dedekind sums (cf. \cite{rademacher} and also \cite{hirzebruchzagierbook} Ch. II $\S8$ for a topological interpretation and generalization) and paramterization of ``exceptional'' bundles on $P^2(\mbb{C})$ (vector bundles $E\to P^2$, with $\operatorname{Hom}(E,E)=\mbb{C}$ and $\operatorname{Ext}^i(E,E)=0$ for $i\geq1$) \cite{rudakov}.  For a readable introduction, references, etc. for this circle of ideas (with a focus on the unicity conjecture), see the recent book of \cite{aignerbook}.  A discussion of the coordinate change between the Teichm\"{u}ller spaces on flat and hyperbolic once-punctured tori (and an interesting differential equation) is given in \cite{keen}.

Of particular background interest is the rate of growth of Markoff numbers, considered in \cite{zagier}, and the generalizations to the rate of growth of lengths of simple closed geodesics on the once-punctured torus (\cite{mcshanerivin1}, \cite{mcshanerivin2}) and arbitrary surfaces (\cite{mirz}).  Given a hyperbolic once-punctured torus $X$, the length function on simple closed geodesics extends to a norm on $H_1(X;\mbb{R})$.  This stems from the fact that each non-trivial homology class in $H_1(X,\mbb{Z})$ has a unique shortest geodesic ``multi-curve'' representative, and primitive homology classes are represented by simple closed geodesics.  Counting simple geodesics of length $\leq L$ is then equivalent to counting primitive lattice points in $L\cdot B_X$, where $B_X$ is the unit ball in this normed space.  See \cite{mcshanerivin1}, \cite{mcshanerivin2}, for details.  The product identity we prove below can be viewed as coming from an analysis of the unit sphere $\partial B_X$.

Long-standing open problems and background aside, below we prove the following product identity over all simple closed geodesics on a once punctured torus.  If $(a,b,c)$, $a,b,c>2$, satisfies $a^2+b^2+c^2=abc$ then
$$
\prod_{\gamma}\left(\frac{t^2}{t^2-4}\right)^h=\frac{a+\sqrt{a^2-4}}{2}\cdot\frac{b+\sqrt{b^2-4}}{2}\cdot\frac{c+\sqrt{c^2-4}}{2},
$$
or rephrased
$$
\prod_{\gamma}\left(\frac{e^{l(\gamma)}+1}{e^{l(\gamma)}-1}\right)^{2h}=\exp\left(\frac{l(\gamma_1)+l(\gamma_2)+l(\gamma_3)}{2}\right)
$$
with ``traces'' $t(\gamma)=2\cosh(l(\gamma)/2)$, $a+b+c=t(\gamma_1)+t(\gamma_2)+t(\gamma_3)$, and $h$ is a positive integer height parameter defined as follows (see Figure \ref{height}).

Embedd a countably infinite 3-regular graph in the plane.  Choose a vertex $v_0$ and and direct all edges away from $v_0$.  Label all three of the regions of the complement of the graph meeting at the chosen vertex with the number 1 (these are the heights corresponding to $a$, $b$, and $c$).  For each directed edge $e$, the region at the target of the edge is to be labeled with the sum of the labels of the two regions meeting at $e$.  We can also label each region with a solution to the Markoff equation (starting with $a$, $b$, $c$ in the regions where $h=1$); the region at the target of the edge $e$ is labeled with $xy-z$ where $x$ and $y$ are the traces to the left and right of the edge, and $z$ is the trace of the region at the source of $e$.  This topograph defines the function $h(\gamma)$.  For an explicit forumla on the projective line $P^1(\mbb{Q})$ (with $a$, $b$, $c$ at $0/1$, $1/1$, $1/0$ in some order), we have $h([p:q])=\max\{|p|,|q|,|p-q|\}$.

\begin{figure}[!hbtp]
\begin{center}
\includegraphics[scale=0.5]{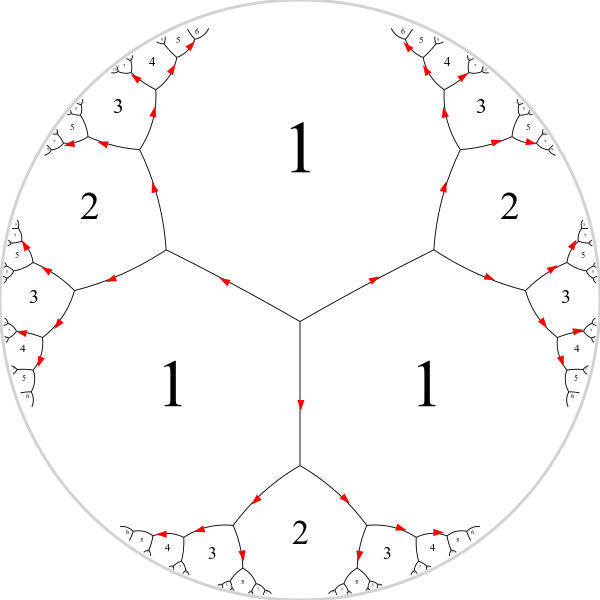}
\end{center}
\caption{A definition by picture of the height $h(p,q)=\max\{|p|,|q|,|p-q|\}$ on the projective line $P^1(\mbb{Q})$.}\label{height}
\end{figure}

The identity we wish to prove is reminiscent of McShane's identity
$$
\sum_{\gamma}\frac{1}{1+e^{l(\gamma)}}=\frac{1}{2}=\sum_{\gamma}\frac{1-\sqrt{1-4/t(\gamma)^2}}{2},
$$
the sum over all simple closed geodesics on a once-punctured torus and $l(\gamma)$ is the length of $\gamma$.  We provide a proof of McShane's identity similar to that of the product identity below in an appendix for comparison.  Both proofs proceed by ``integrating'' around a chosen point in the Teichm\"{u}ller modular orbit of a once-punctured torus, but neither requires any specialized geometric knowledge.
\section*{Acknowledgments}
The author thanks Martin Weissman for allowing the use of his well-commented code to produce Figure \ref{height} and Greg McShane for providing references concerning the length-norm on homology.
%%%%%%%%%%%%%%%%%%%%%%%%%%%%%%%%%%%
%%%%%%%%%%%%%%%%%%%%%%%%%%%%%%%%%%%
%%%%%%%%%%%%%%%%%%%%%%%%%%%%%%%%%%%
\section{Proof of the product identity}
Choose real numbers $a$, $b$, $c>2$ satisfying $a^2+b^2+c^2=abc$.  We will construct the infinite product from three pieces identical in form, one for each pair $(a,b)$, $(b,c)$, $(c,a)$, so we focus on the construction of the $(a,b)$ piece.  For each $p/q\in\mbb{Q}\cap[0,1]\cup{1/0}$, define $t_{p/q}$ recursively by
$$
t=t_{p/q}=t_{p'/q'}t_{p''/q''}-t_{\tilde{p}/\tilde{q}}, \ t_{0/1}=a, \ t_{1/1}=b, \ t_{1/0}=c
$$
where $p'/q'<p''/q''$ are adjacent rationals ($p''q'-p'q''=1$), $p/q=(p'+p'')/(q'+q'')$ is their Farey sum, and $\tilde{p}/\tilde{q}$ is the other rational adjacent to both $p'/q'$ and $p''/q''$ (see Figure \ref{decorate}).  Let
$$
l=l(p/q)=2\operatorname{arccosh}(t_{p/q}/2)=2\log\left(\frac{t+\sqrt{t^2-4}}{2}\right)
$$
and define
$$
F:\mbb{Q}\cap[0,1]\to\mbb{R}, \ F(p/q)=\frac{l(p/q)}{q}.
$$
See Figure \ref{fock} for four instances of this function.

Although not required to follow the proof, we can relate this function to the length-norm unit ball of $H_1(X;\mbb{R})\cong\mbb{R}^2$ ($X$ the once-punctured torus) as follows.  Suppose $a=\operatorname{tr}(A)$, $b=\operatorname{tr}(B)$, and $c=\operatorname{tr}(C)$ are the (postive) traces of hyperbolic elements $A$, $B$, $C\in PSL_2(\mbb{R})$ (with $\pi_1(X)\cong\langle A,B\rangle$ free on two generators) projecting to the representatives of the homology elements $(1,0)$, $(1,1)$, and $(0,1)$ (i.e. we've chosen a basis for homology).  Then the coordinates of points on the unit ball $B_X$ coming from primitive integral homology classes $(q,p)$ with $0\leq p\leq q$ are $\displaystyle{\left(\frac{q}{l(p/q)},\frac{p}{l(p/q)}\right)}$.  In other words, $F$ above is the reciprocal of the $x$-coordinate as a function of slope on the part of the unit ball corresponding to slopes between zero and one.  See Figure \ref{fock} for some instances of the length-norm unit ball.  The function $F$ was also considered in \cite{fockgoncharov}, where convexity was noted as well.  Smoothness of $F$ at irrational $x$ can also be deduced from \cite{mcshanerivin1}, where some properties of the unit ball are demonstrated and conjectured.

\begin{figure}[!hbtp]
\begin{center}
\includegraphics[scale=0.3]{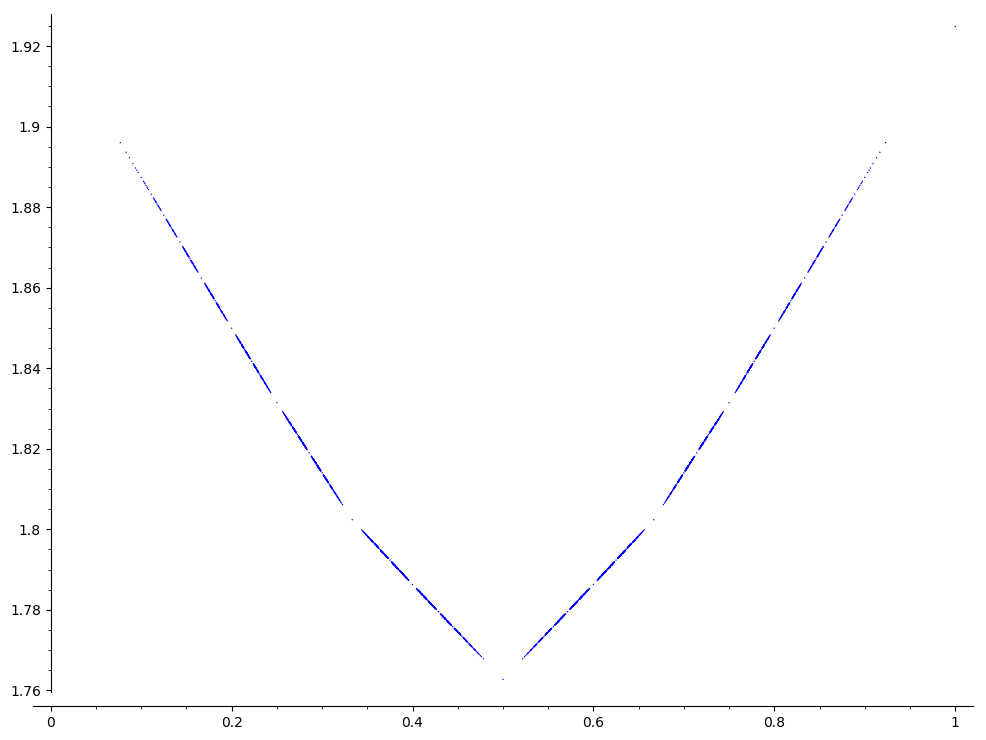}
\includegraphics[scale=0.3]{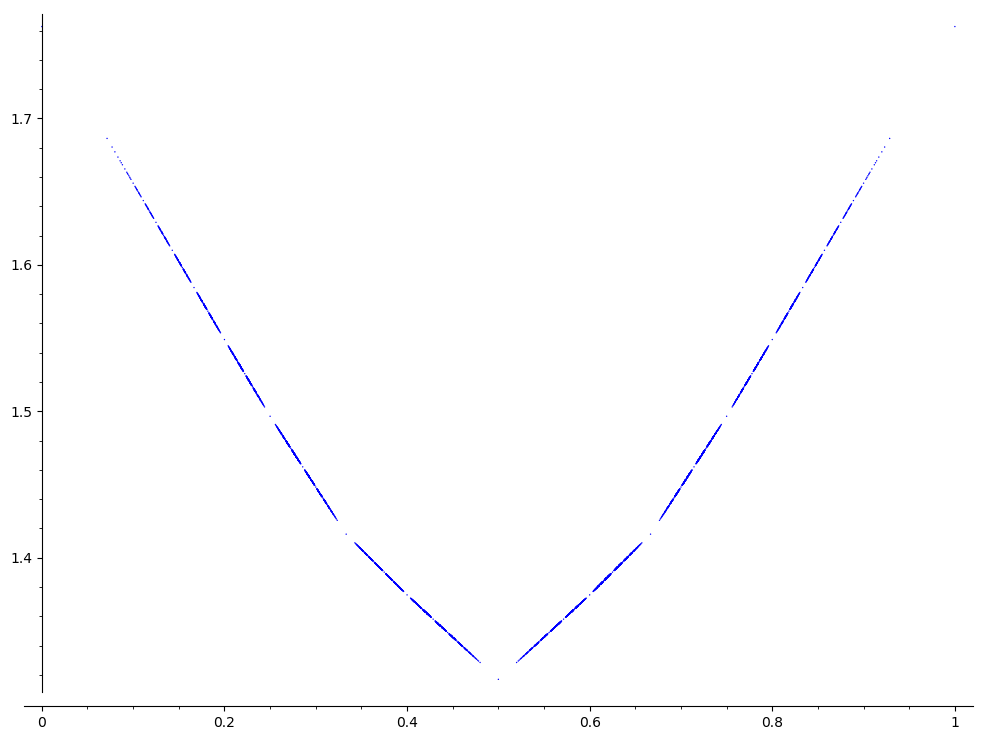}
\includegraphics[scale=0.3]{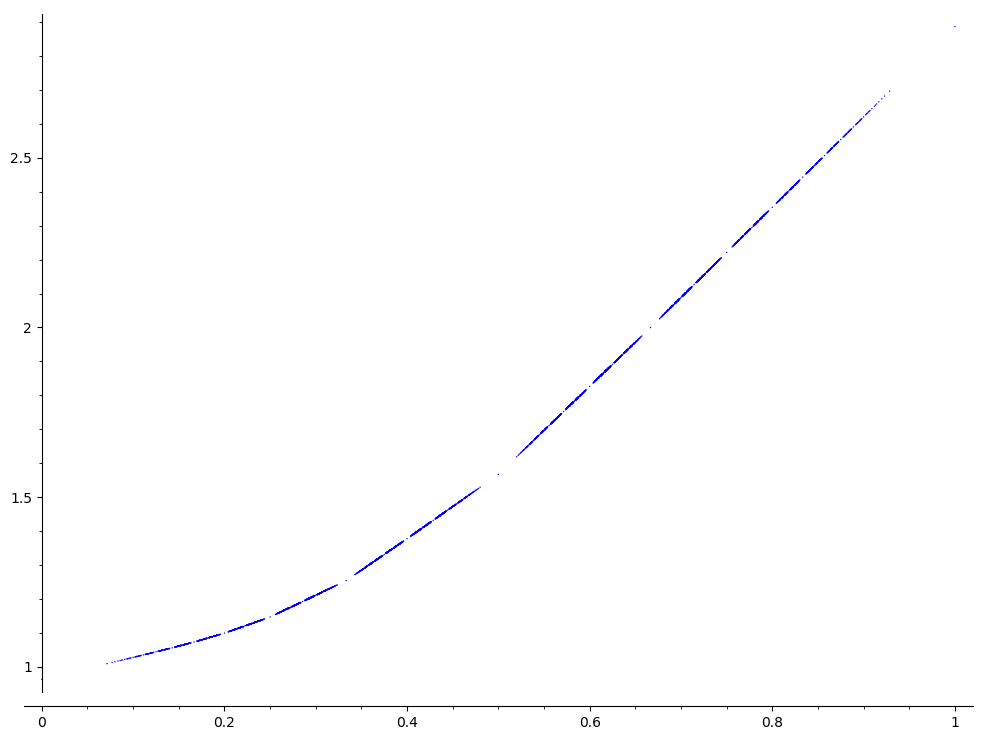}
\includegraphics[scale=0.3]{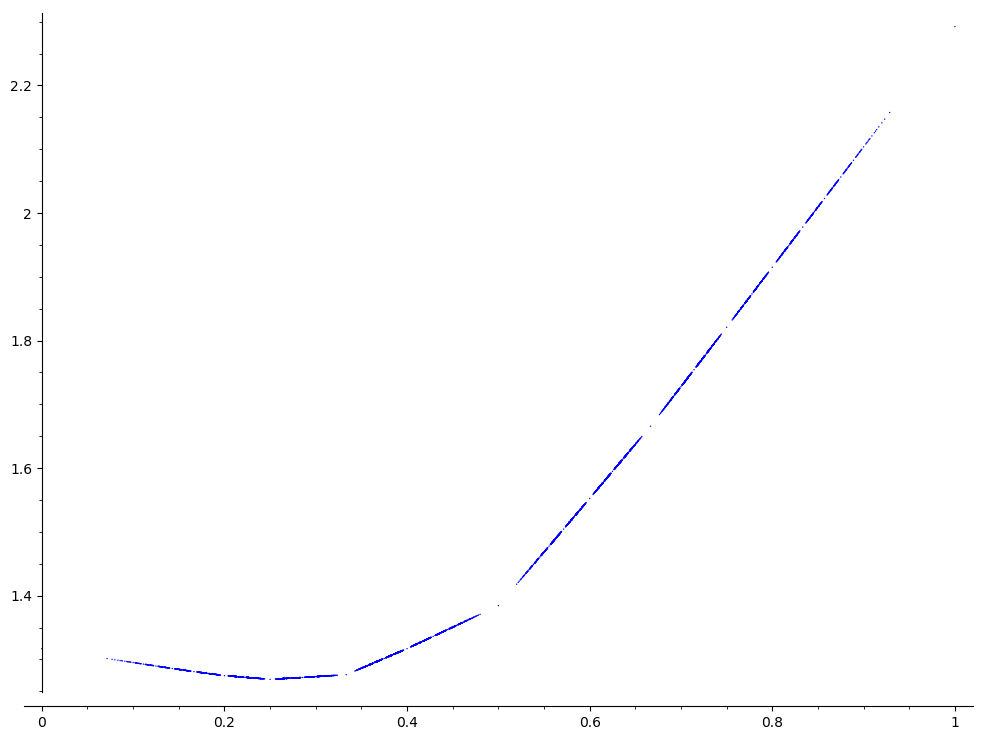}
\includegraphics[scale=0.3]{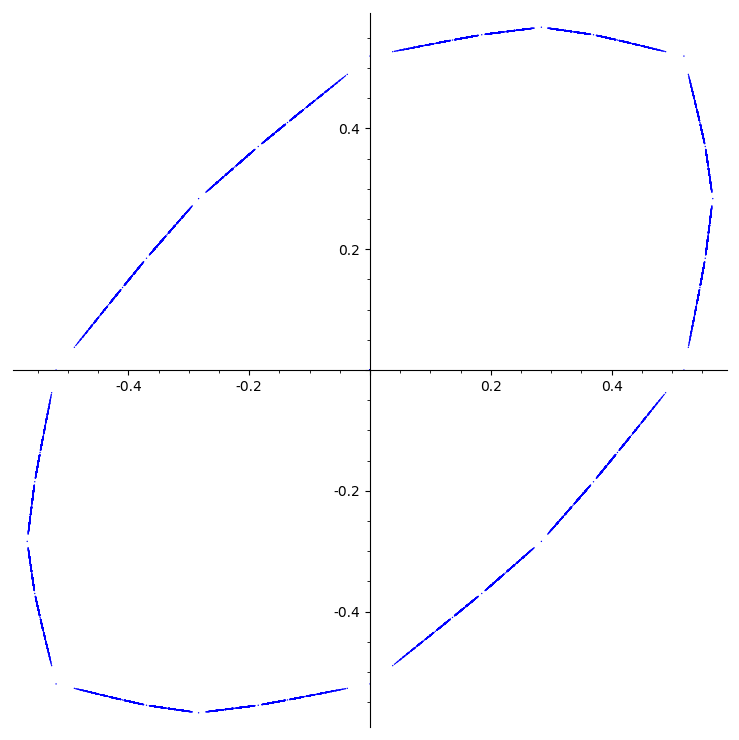}
\includegraphics[scale=0.3]{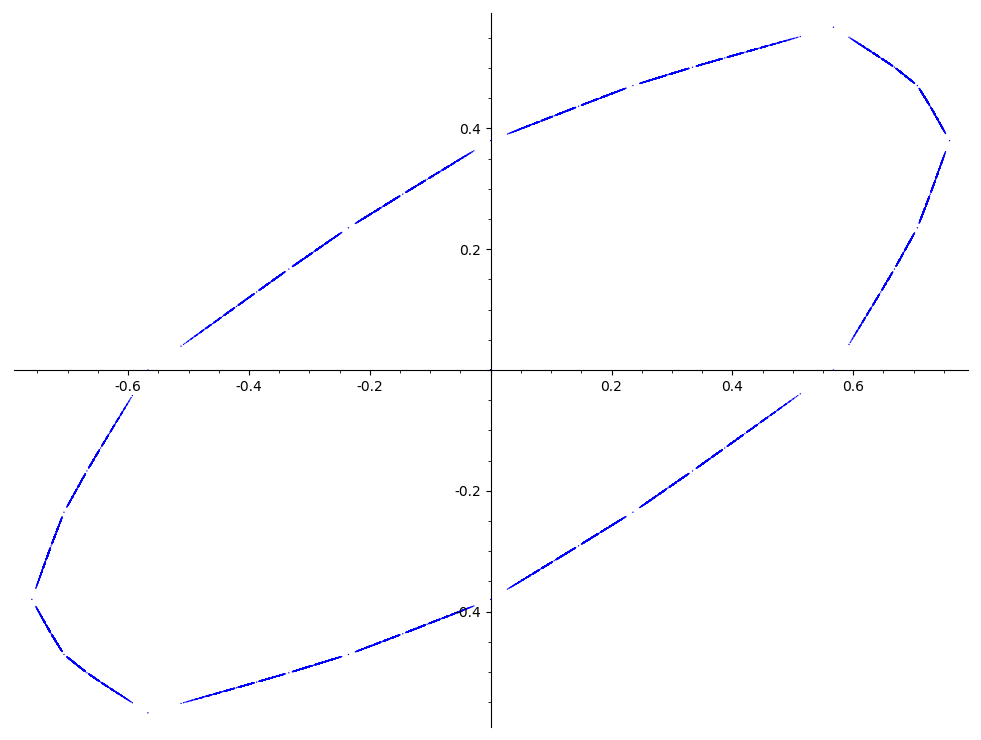}
\includegraphics[scale=0.3]{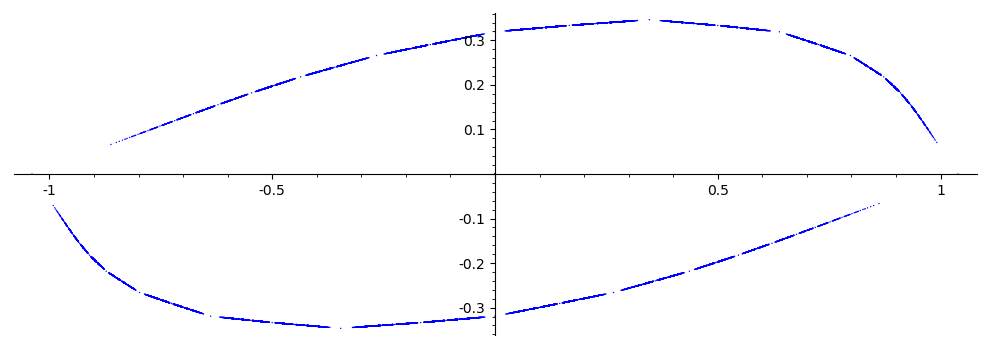}
\includegraphics[scale=0.3]{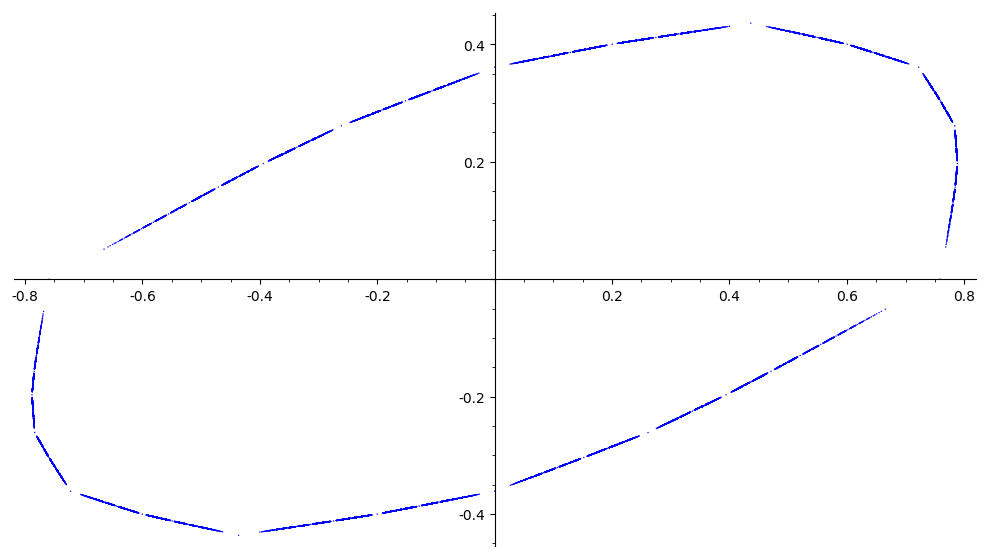}
\captionsetup{justification=raggedright,format=hang}
\caption{The $(a,b)$ version of $F$ and length-norm unit ball for $(a,b,c)$=$(3,3,3)$, $(\sqrt{8},\sqrt{8},4)$, $(\sqrt{5},\sqrt{20},5)$, $(\sqrt{6},\sqrt{12},\sqrt{18})$ (left to right, top to bottom).}\label{fock}
\end{center}
\end{figure}

In what follows, we will adhere to the convention that undecorated variables correspond to their values at $p/q$ and their decorated versions correspond to their values at $p'/q'$, $p''/q''$, $\tilde{p}/\tilde{q}$ (e.g. $t'=t_{p'/q'}$ if $t=t_{p/q}$); see Figure \ref{decorate}.

\begin{figure}[!hbtp]
\begin{center}
\begin{tikzpicture}[scale=2]
\draw[->-=.5, thick](0,1)--(0,0);
\draw[->-=.5, thick](0,0)--({-sqrt(3)},-1);
\draw[->-=.5, thick](0,0)--({sqrt(3)},-1);

\draw[->-=.5, thick]({-sqrt(3)},2)--(0,1);
\draw[->-=.5, thick]({sqrt(3)},2)--(0,1);

%\draw(0,-0.5)node{$t=t_{p/q}$};
\draw(0,-0.5)node{$\frac{p}{q}=\frac{p'+p''}{q'+q''}$};

%\draw(0,1.5)node{$t=t_{p/q}$};
%\draw(0.75,0.5)node{$t''=t_{p''/q''}$};
%\draw(-0.75,0.5)node{$t'=t_{p'/q'}$};

\draw(0,1.5)node{$\tilde{p}/\tilde{q}$};
\draw(0.75,0.5)node{$p''/q''$};
\draw(-0.75,0.5)node{$p'/q'$};
\end{tikzpicture}

\caption{Decorated variable conventions}\label{decorate}
\end{center}
\end{figure}
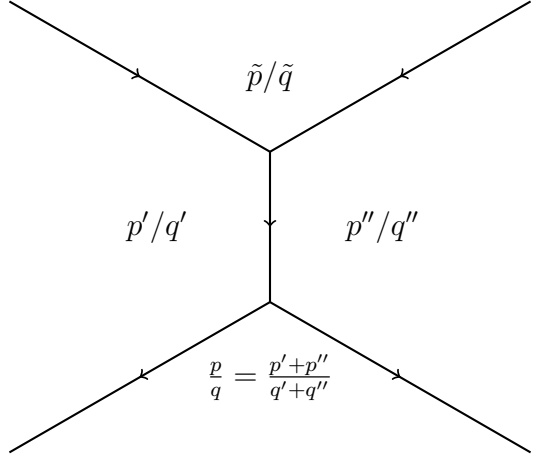

\begin{theorem*}
The function $F$ is convex on $\mbb{Q}\cap[0,1]$ (hence extends to a continuous function $F$ on $[0,1]$).  $F$ is differentiable at irrational $x$, with left and right derivatives at $p/q$ given by
\begin{align*}
\lambda(p/q):&=\lim_{x\to p/q^-}\frac{F(p/q)-F(x)}{p/q-x}=-\log\left(\left(\frac{t'-t''e^{-l/2}}{e^{l/2}-e^{-l/2}}\right)^{q}e^{lq''}\right),\\
\rho(p/q):&=\lim_{x\to p/q^+}\frac{F(p/q)-F(x)}{p/q-x}=\log\left(\left(\frac{t''e^{l/2}-t'}{e^{l/2}-e^{-l/2}}\right)^{q}e^{-lq''}\right).
\end{align*}
\end{theorem*}
\begin{proof}
First note that $l<l'+l''$, which follows from the triangle inequality in two-dimensional hyperbolic space or by a direct argument.  Consider the difference quotients from the left and right at a triple of adjacent rationals $p'/q'<p/q<p''/q''$.  We have (with primes denoting function values per our convention and not derivatives)
\begin{align*}
\frac{F-F'}{p/q-p'/q'}&=q'l-ql'=q'l-(q'+q'')l'<q'(l'+l'')-(q'+q'')l'=q'l''-q''l',\\
\frac{F''-F}{p''/q''-p/q}&=ql''-q''l=(q'+q'')l''-q''l>(q'+q'')l''-q''(l'+l'')=q'l''-q''l',\\
&\Rightarrow\frac{F-F'}{p/q-p'/q'}<\frac{F''-F'}{p''/q''-p'/q'}<\frac{F''-F}{p''/q''-p/q},
\end{align*}
i.e. the adjacent slope to the left at $p/q$ is less than the adjacent slope to the right of $p/q$.  Two non-adjacent rationals can be favorably compared along a chain of adjacent rationals, proving convexity.

We want to take limits of the left and right slopes along the rationals adjacent to $p/q$.  So as not to break our notational conventions, we fix $p'/q'<p/q<p''/q''$ (and the various decorated and undecorated variables) and let $n$ be the ``dynamic'' variable in what follows.  The rationals adjacent to $p/q$ are
$$
x_n=\frac{p_n}{q_n}=\frac{|p''+np|}{|q''+nq|}, \ n\in\mbb{Z},
$$
and the corresponding slope between $x_n$ and $p/q$ is
$$
S_n=\frac{F(x_n)-F(p/q)}{x_n-p/q}=\left\{
\begin{array}{cc}
ql_n-lq_n&n\geq0\\
lq_n-ql_n&n\leq-1\\
\end{array}\right.,
$$
where $l_n=l(x_n)$.  The traces $t_n=t_{x_n}$ satisfy a linear recurrence
$$
t_{n+2}=tt_{n+1}-t_{n}, \ t_{-1}=t', \ t_0=t''
$$
with solution
$$
t_n=\frac{t''e^{l/2}-t'}{e^{l/2}-e^{-l/2}}e^{nl/2}+\frac{t'-t''e^{-l/2}}{e^{l/2}-e^{-l/2}}e^{-nl/2}, \ e^{l/2}=\frac{t+\sqrt{t^2-4}}{2}.
$$
We have (for $n\geq0$)
\begin{align*}
\exp(S_n)&=\exp(ql_n-lq_n)=\exp(ql_n-l(q''+nq))=\left(\frac{e^{l_n}}{e^{nl}}\right)^qe^{-lq''}\\
&=\left(\frac{t_n+\sqrt{t_n^2-4}}{2}\right)^{2q}e^{-nlq}e^{-lq''}\xrightarrow{n\to\infty}\left(\frac{t''e^{l/2}-t'}{e^{l/2}-e^{-l/2}}\right)^qe^{-lq''},
\end{align*}
and (for $n\leq-1$)
\begin{align*}
\exp(S_n)&=\exp(lq_n-ql_n)=\exp(l|q''+nq|-ql_n)=\left(\frac{e^{-nl}}{e^{l_n}}\right)^qe^{-lq''}\\
&=\left(\frac{2}{t_n+\sqrt{t_n^2-4}}\right)^{2q}e^{-nlq}e^{-lq''}\xrightarrow{n\to-\infty}\left(\frac{t'-t''e^{-l/2}}{e^{l/2}-e^{-l/2}}\right)^{-q}e^{-lq''}.
\end{align*}

Finally, we need to address differentiability of $F$ at irrational $x$.  For a triple of adjacent rationals, the difference in right and left slopes is
$$
\frac{F''-F}{p''/q''-p/q}-\frac{F-F'}{p/q-p'/q'}=q(l''+l'-l).
$$
We want to show that this quantity approaches zero as $p/q\to x\in[0,1]\backslash\mbb{Q}$.  To do so, we need to strengthen the triangle inequality.  The argument that follows is a variation of \cite{zagier}, Lemma 2.  In the Fricke trace identity, we get additivity of lengths when $A$ and $B$ commute so that
$$
2\operatorname{arccosh}(z/2)=2\operatorname{arccosh}(x/2)+2\operatorname{arccosh}(y/2)\Longleftrightarrow x^2+y^2+z^2=xyz+4
$$
(which can also be verified directly).  With $\displaystyle{\hat{t}=2\cosh\left(\frac{l'+l''}{2}\right)}$ we have (subtracting the Markoff equation from the ``additive'' Fricke equation)
$$
\hat{t}^2-t^2=t't''(\hat{t}-t)+4=(t+\tilde{t})(\hat{t}-t)+4
$$
and therefore
$$
4=(\hat{t}-t)(\hat{t}+\tilde{t})>(\hat{t}-t)t.
$$
Applying the monotone function $2\operatorname{arccosh}(w/2)$ to the inequality
$$
\hat{t}<t+\frac{4}{t}
$$
gives
$$
l'+l''<2\operatorname{arccosh}\left(\frac{t}{2}+\frac{2}{t}\right)\leq l+\frac{4}{t\sqrt{t^2-4}}=l+\frac{2}{\sinh l}
$$
using a tangent line approximation.  Returning to our considerations above, we have
$$
0\leq q(l'+l''-l)\leq\frac{2q}{\sinh l}\xrightarrow{q, l\to\infty}0
$$
for rationals $p/q\to x$.  The limit is zero since, for instance, $q/l\to 1/F(x)$ converges and $l/\sinh l\to0$.
\end{proof}
While the left and right derivatives are somewhat complicated, depending on triples of values near a vertex, their difference is simpler
$$
\exp(\rho-\lambda)=\left(\frac{(t'-t''e^{-l/2})(t''e^{l/2}-t')}{(e^{l/2}-e^{-l/2})^2}\right)^q=\left(\frac{t^2}{t^2-4}\right)^q=\coth^{2q}(l/2)=\left(\frac{e^l+1}{e^l-1}\right)^{2q}.
$$
Starting with $\rho(0/1)$ and summing over the jumps in the derivative of $F$ at rationals in $(0,1)$ gives (in exponential form)
$$
\exp(\rho(0/1))\prod_{p/q\in(0,1)}\coth^{2q}(l/2)=\exp(\lambda(1/1)).
$$ 
One can check that the expressions given for $\rho$, $\lambda$ hold at $0/1$, $1/1$, respectively with the convention that $t_{1/0}=c$ and $q''(0/1)=1$, $q''(1/1)=0$, so the above can be written as
$$
\prod_{p/q\in(0,1)}\coth^{2q}(l/2)=\exp(\lambda(1/1)-\rho(0/1))=\frac{\sqrt{a^2-4}\sqrt{b^2-4}}{\left(a-c\frac{b-\sqrt{b^2-4}}{2}\right)\left(b-c\frac{a-\sqrt{a^2-4}}{2}\right)}
$$
To cover all of the ``topograph'' (and account for all simple closed geodesics), we repeat the construction above twice more, replacing the pair $(a,b)$ with $(b,c)$ and $(c,a)$, and include terms for $a$, $b$, and $c$, themselves.  We will now replace the variable $q$ in the product with the height $h$, where $h$ is defined as in Figure \ref{height}.  The completed product is (using the relation $a^2+b^2+c^2=abc$ a few times to simplify)
\begin{align*}
\prod_{\gamma}\coth^{2h}(l(\gamma)/2)&=\frac{a^2}{a^2-4}\frac{b^2}{b^2-4}\frac{c^2}{c^2-4}&\text{ terms for the base triple } a, b, c\\
&\cdot\frac{\sqrt{a^2-4}\sqrt{b^2-4}}{\left(a-c\frac{b-\sqrt{b^2-4}}{2}\right)\left(b-c\frac{a-\sqrt{a^2-4}}{2}\right)}&\text{ term for the pair } a, b\\
&\cdot\frac{\sqrt{b^2-4}\sqrt{c^2-4}}{\left(b-a\frac{c-\sqrt{c^2-4}}{2}\right)\left(c-a\frac{b-\sqrt{b^2-4}}{2}\right)}&\text{ term for the pair } b, c\\
&\cdot\frac{\sqrt{c^2-4}\sqrt{a^2-4}}{\left(c-b\frac{a-\sqrt{a^2-4}}{2}\right)\left(a-b\frac{c-\sqrt{c^2-4}}{2}\right)}&\text{ term for the pair } c, a\\
&=\frac{a+\sqrt{a^2-4}}{2}\cdot\frac{b+\sqrt{b^2-4}}{2}\cdot\frac{c+\sqrt{c^2-4}}{2}.
\end{align*}
See Figure \ref{random} for an example of the three $F$ functions for a random triple.
\begin{figure}[!hbtp]
\begin{center}
\includegraphics[scale=0.2]{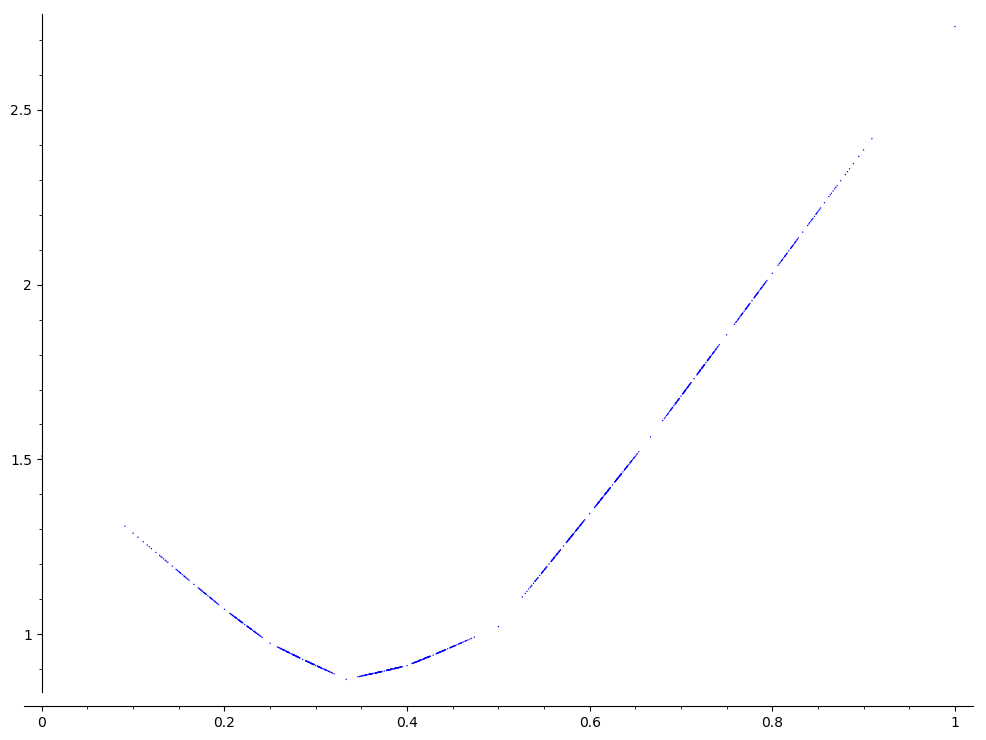}
\includegraphics[scale=0.2]{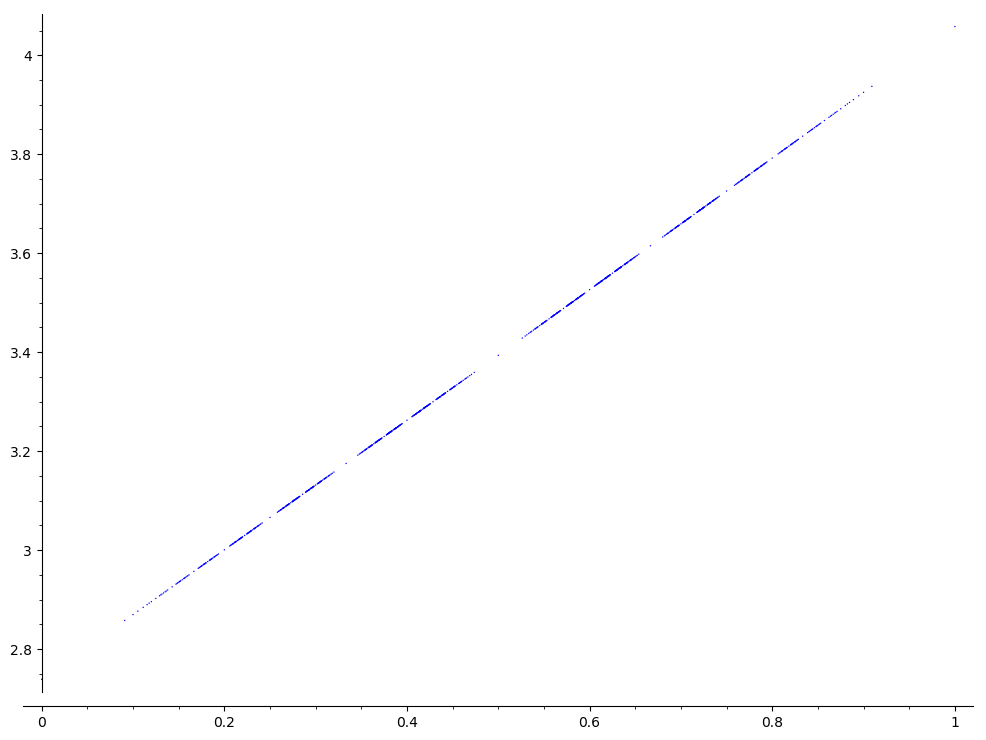}
\includegraphics[scale=0.2]{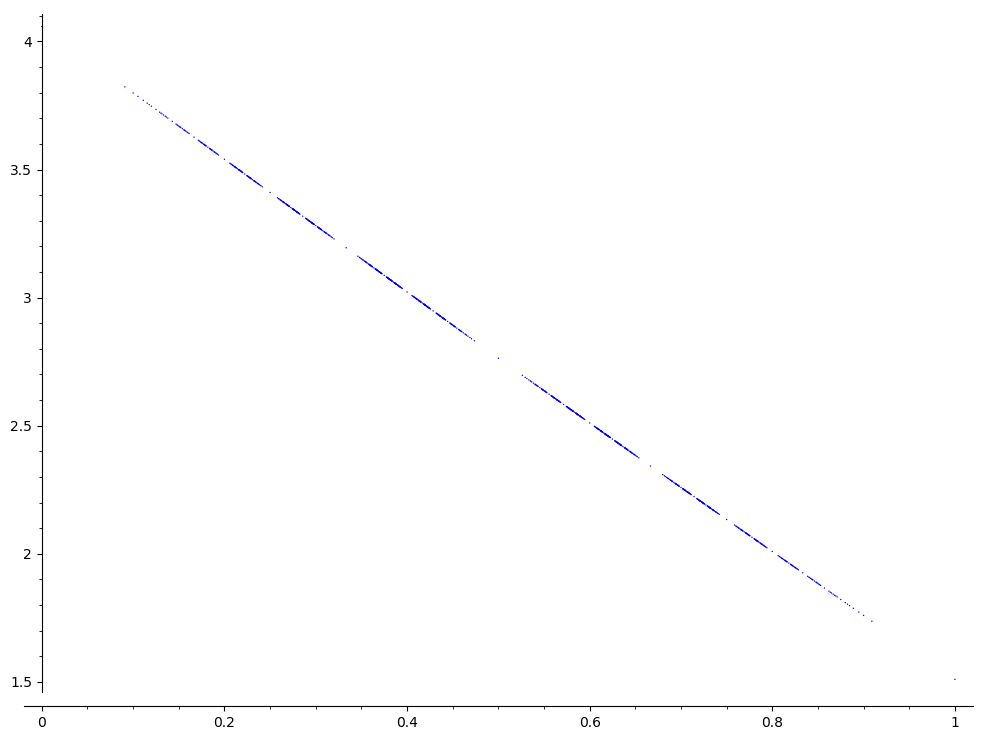}
\captionsetup{justification=raggedright,format=hang}
\caption{The $(a,b)$, $(b,c)$, and $(c,a)$ versions of $F$ for a random triple $(a,b,c)=(2.59740058623, 4.18711171215, 7.73808785195)$.}\label{random}
\end{center}
\end{figure}
%%%%%%%%%%%%%%%%%%%%%%%%%%%%%%%%%%%
%%%%%%%%%%%%%%%%%%%%%%%%%%%%%%%%%%%
%%%%%%%%%%%%%%%%%%%%%%%%%%%%%%%%%%%
\section{Example:  The modular torus at its point of symmetry}
For the maximally symmetric torus, choosing $(a,b,c)=(3,3,3)$, the product is (factors up to $h=7$ shown)
\begin{align*}
\left(\frac{3+\sqrt{5}}{2}\right)^3=&\left(\frac{1^2}{1^2-4/9}\right)^{3\cdot 1}\left(\frac{2^2}{2^2-4/9}\right)^{3\cdot2}\left(\frac{5^2}{5^2-4/9}\right)^{6\cdot3}\left(\frac{13^2}{13^2-4/9}\right)^{6\cdot4}\\
&\left(\frac{29^2}{29^2-4/9}\right)^{6\cdot5}\left(\frac{34^2}{34^2-4/9}\right)^{6\cdot5}\left(\frac{89^2}{89^2-4/9}\right)^{6\cdot6}\left(\frac{169^2}{169^2-4/9}\right)^{6\cdot7}\\
&\left(\frac{194^2}{194^2-4/9}\right)^{6\cdot7}\left(\frac{233^2}{233^2-4/9}\right)^{6\cdot7}\dots
\end{align*}
where the numbers
$$
\{1,2,5,13,29,34,89,169,194,233,\ldots\}
$$
(as a multiset) are the largest elements of triples of positive integer solutions to $x^2+y^2+z^2=3xyz$.  In this case the traces are all divisible by $3$, hence the factor of $3$ in the ``number theory'' version.  Each factor occurs with either multiplicity $3$ (for the ``singular'' Markoff numbers $1$ and $2$) or multiplicity $6$ (for the rest), i.e. the exponent $6\cdot h$ continues indefinitely.

Similarly, one could start with any other Markoff triple to obtain a similar infinite product of rational numbers whose value is the product of three units each generating (the norm 1 cyclic part of) the stabilizers of the Markoff forms.  Or, starting with traces in a number field, one obtains similar infinite products for triple products of relative quadratic units.  The author doesn't know whether or not these products are more than curiosities.
%%%%%%%%%%%%%%%%%%%%%%%%%%%%%%%%%%%
%%%%%%%%%%%%%%%%%%%%%%%%%%%%%%%%%%%
%%%%%%%%%%%%%%%%%%%%%%%%%%%%%%%%%%%
\section{Appendix:  McShane's identity}

In this appendix, we give a proof of McShane's identity
$$
\sum_{\gamma}\frac{1}{1+e^{l(\gamma)}}=\frac{1}{2}=\sum_{\gamma}\frac{1-\sqrt{1-4/t(\gamma)^2}}{2}
$$
by considering the sum over the jump discontinuities in a function to be defined below.  First choose a triple $(a,b,c)$, $a,b,c>2$, satisfying $a^2+b^2+c^2=abc$.  [Once again we will patch together three peices to create the sum above, one for each of $(a,b)$, $(b,c)$, and $(c,a)$, so we'll focus on the $(a,b)$ piece.]  Associated to this choice, we can choose generators (pairwise) for the associated once-punctured torus (cf. \cite{schmidt})
$$
M_{0/1}=\left(
\begin{array}{cc}
c/b&a/b^2\\
a&a-c/b\\
\end{array}
\right), \ 
M_{1/1}=\left(
\begin{array}{cc}
0&-1/b\\
b&b\\
\end{array}
\right), \ 
M_{1/0}=M_{0/1}M_{1/1}^{-1}=\left(
\begin{array}{cc}
c-a/b&c/b^2\\
c&a/b\\
\end{array}
\right).
$$
The reader (or their favorite computer algebra system) can verify that the traces of the pairwise commutators are $-2$ for this triple.  For $p/q\in(0,1)\cap\mbb{Q}$ inductively define matrices
$$
M_{p/q}=M_{p'/q'}M_{p''/q''}, \ \frac{p'}{q'}<\frac{p}{q}=\frac{p'+p''}{q'+q''}<\frac{p''}{q''}, \ p''q'-p'q''=1.
$$
Let $t_{p/q}=\operatorname{tr}(M_{p/q})$, and continue our notational conventions regarding $t'$, $t''$.  For each adjacent triple, the traces satisfy the Markoff--Fricke equation $t^2+t'^2+t''^2=tt't''$.  We will use the notation
$$
M_{p/q}=\left(
\begin{array}{cc}
u&u-v\\
t&t-u\\
\end{array}
\right), \ v=\frac{1+u^2}{t},
$$
for the entries of each matrix (with our primed variable conventions for $u$ and $v$ as well).  One can verify inductively starting with $M_{0/1}$, $M_{1/1}$ that the lower left corner of each matrix is the trace of the matrix.  Multiplying $M_{p'/q'}$ and $M_{p''/q''}$ gives the recurrences
$$
u=u'u''+u't''-v't'', \ t=t'u''+t't''-u't''.
$$
We now consider the function (see Figure \ref{mcshane} for an example graph)
$$
f:[0,1]\cap\mbb{Q}\to\mbb{R}, \ p/q\mapsto u/t.
$$

\begin{figure}[!hbtp]
\begin{center}
\includegraphics[scale=0.3]{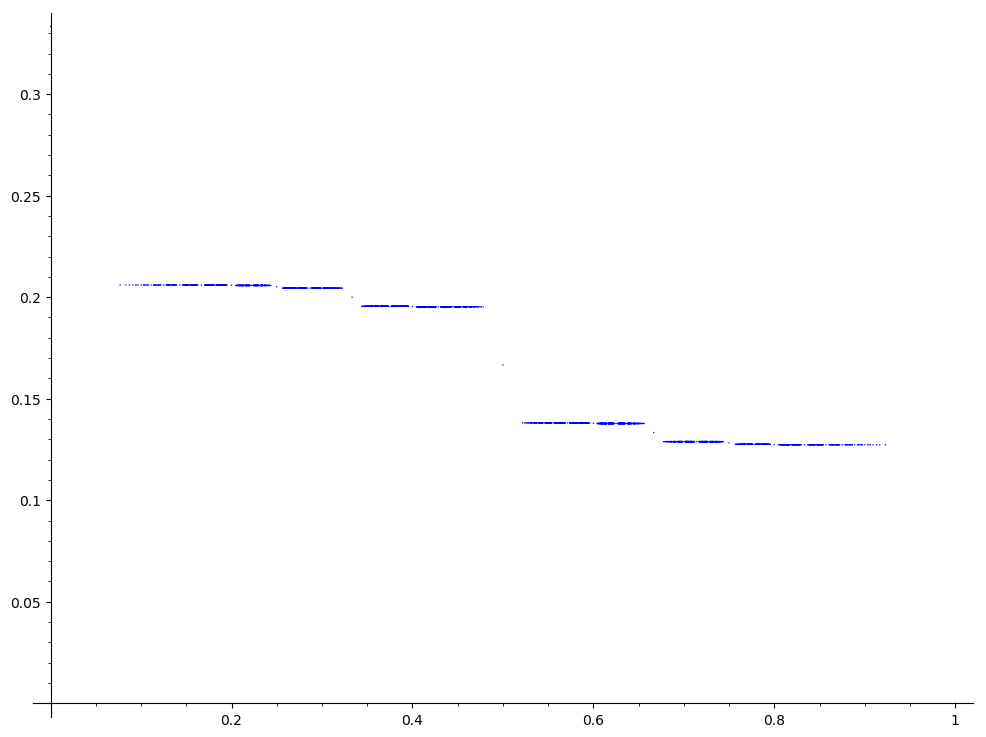}
\includegraphics[scale=0.3]{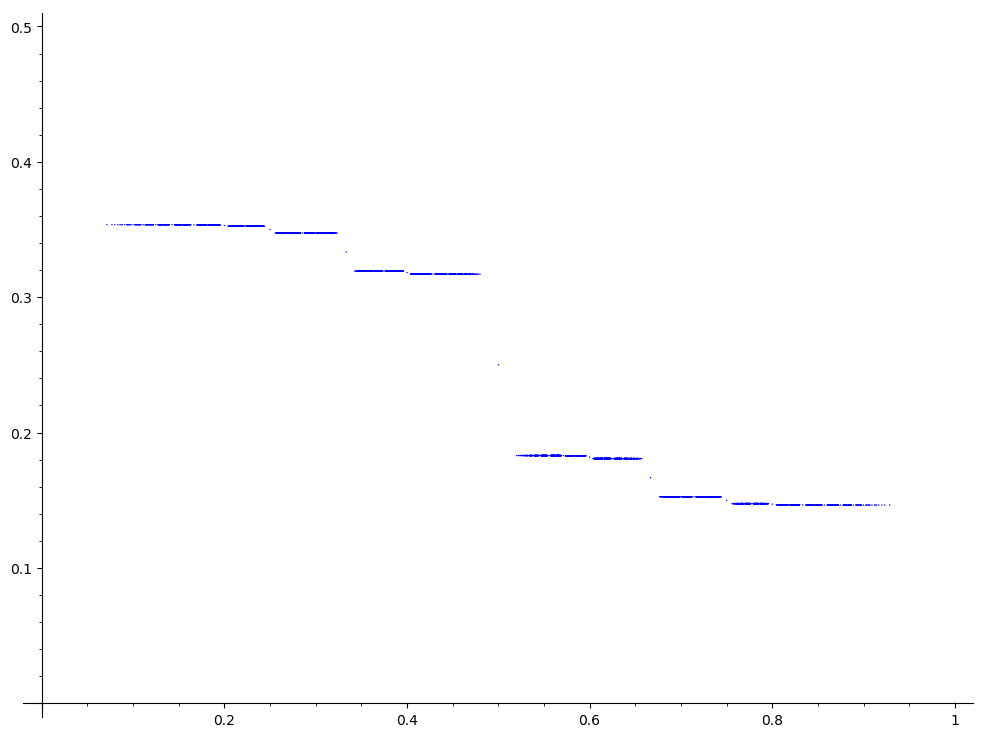}
\includegraphics[scale=0.3]{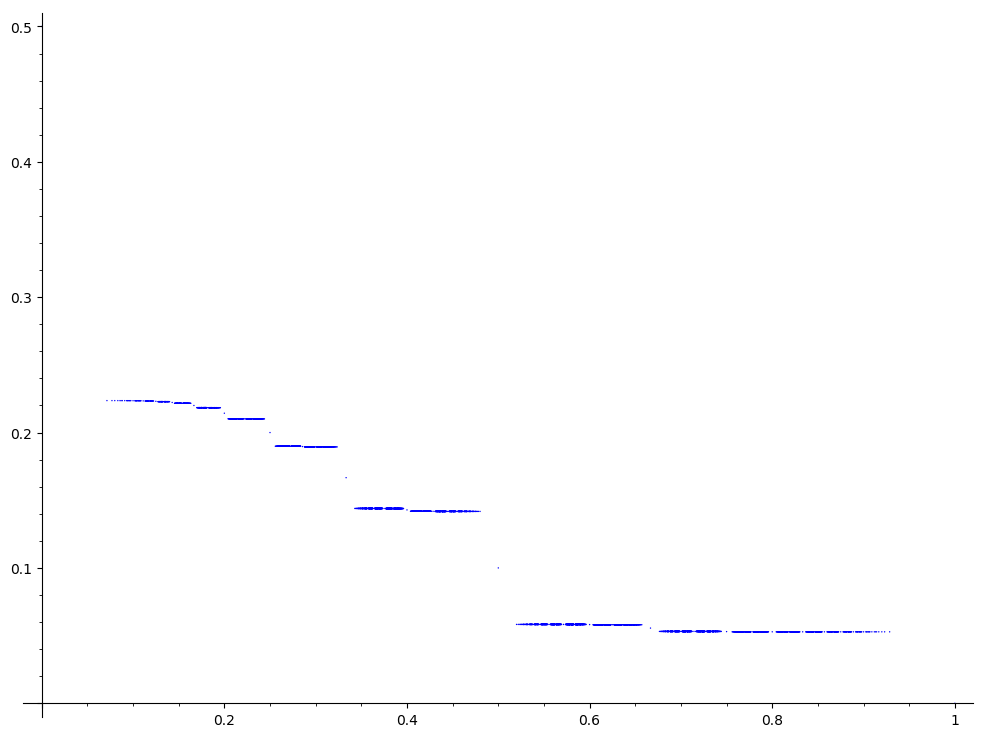}
\includegraphics[scale=0.3]{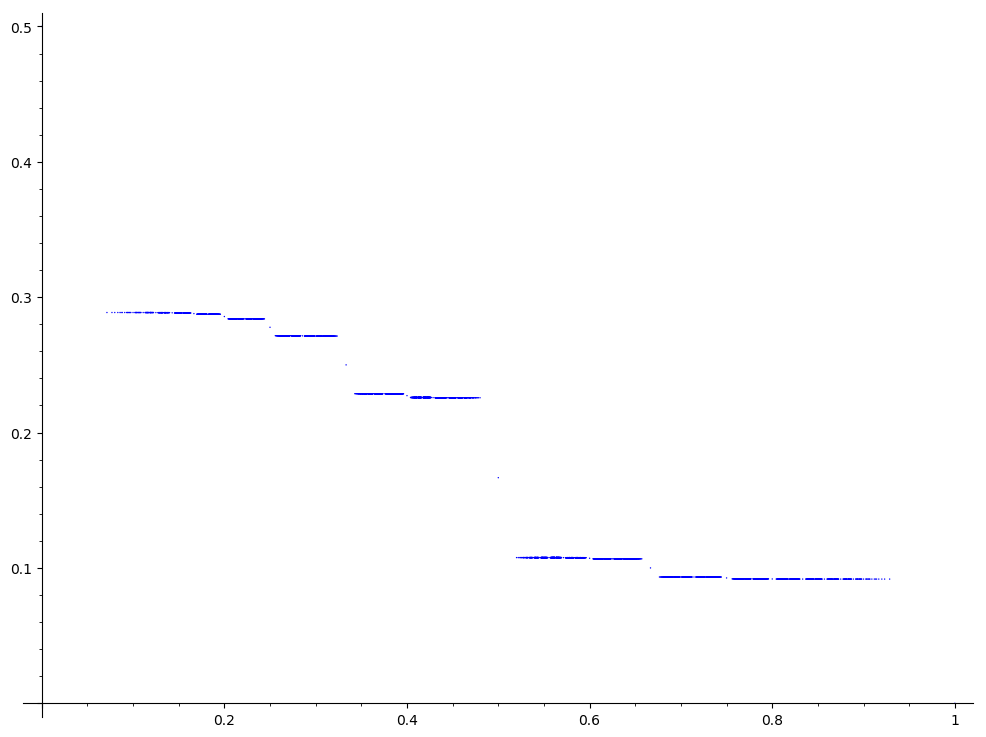}
\captionsetup{justification=raggedright,format=hang}
\caption{The $(a,b)$ version of $f$ for $(a,b,c)$=$(3,3,3)$, $(\sqrt{8},\sqrt{8},4)$, $(\sqrt{5},\sqrt{20},5)$, $(\sqrt{6},\sqrt{12},\sqrt{18})$ (left to right, top to bottom).}\label{mcshane}
\end{center}
\end{figure}

The function $f$ is decreasing (as we'll see), continuous at irrationals, with jumps on both sides at rationals.  The difference between adjacent rationals is
\begin{align*}
f(p''/q'')-f(p'/q')=\frac{u''}{t''}-\frac{u'}{t'}=\frac{u''t'-u't''}{t't''}=\frac{t}{t't''}-1<0,
\end{align*}
using the recurrence for $t$ noted above.  (That the difference is negative follows from the Markoff equation.)  Using the Markoff equation again, we have
$$
f(p''/q'')-f(p'/q')=\frac{-1+\sqrt{1-4\left(\frac{1}{t'^2}+\frac{1}{t''^2}\right)}}{2}.
$$
Now we compute right and left limits:
\begin{align*}
\lim_{\frac{p''}{q''}\to\frac{p'}{q'}^+}f(p''/q'')-f(p'/q')&=\lim_{t''\to\infty}\frac{-1+\sqrt{1-4\left(\frac{1}{t'^2}+\frac{1}{t''^2}\right)}}{2}=\frac{-1+\sqrt{1-4/t'^2}}{2},\\
\lim_{\frac{p'}{q'}\to\frac{p''}{q''}^-}f(p''/q'')-f(p'/q')&=\lim_{t'\to\infty}\frac{-1+\sqrt{1-4\left(\frac{1}{t'^2}+\frac{1}{t''^2}\right)}}{2}=\frac{-1+\sqrt{1-4/t''^2}}{2}.
\end{align*}
In other words, the jump in $f$ at $p/q$ is given by
\begin{align*}
\lim_{x\to p/q^+}f(x)-\lim_{x\to p/q^-}=-1+\sqrt{1-4/t^2}.
\end{align*}
At irrational $x$, taking a sequence of rationals $p'/q'\shortrightarrow x\shortleftarrow p''/q''$ (say along the slow continued fraction expansion of $x$) gives
$$
\lim_{p'/q'\shortrightarrow x\shortleftarrow p''/q''}f(p''/q'')-f(p'/q')=\lim_{t', t''\to\infty}\frac{-1+\sqrt{1-4\left(\frac{1}{t'^2}+\frac{1}{t''^2}\right)}}{2}=0,
$$
proving continutity of $f$ at irrational $x$.

Starting at zero and adding the jumps in $f$ at rationals (including the ``half-jumps'' at 0 and 1) gives
$$
\frac{c}{ab}+\frac{-1+\sqrt{1-4/a^2}}{2}+\sum_{p/q\in(0,1)}\left(-1+\sqrt{1-4/t^2}\right)+\frac{-1+\sqrt{1-4/b^2}}{2}=0.
$$
Repeating the argument above three times and throwing in the terms for $a$, $b$, and $c$ themselves gives
\begin{align*}
\sum_{\gamma}\left(-1+\sqrt{1-4/t^2}\right)&=-3+\sqrt{1-4/a^2}+\sqrt{1-4/a^2}+\sqrt{1-4/a^2}&\text{for } a, b, c\\
&+\frac{1-\sqrt{1-4/a^2}}{2}+\frac{1-\sqrt{1-4/b^2}}{2}-\frac{c}{ab}&\text{for the pair } (a,b)\\
&+\frac{1-\sqrt{1-4/b^2}}{2}+\frac{1-\sqrt{1-4/c^2}}{2}-\frac{a}{bc}&\text{for the pair } (b,c)\\
&+\frac{1-\sqrt{1-4/c^2}}{2}+\frac{1-\sqrt{1-4/a^2}}{2}-\frac{b}{ca}&\text{for the pair } (c,a)\\
&=-1&\text{(Markoff equation)}.
\end{align*}
This gives McShane's identity after noting that
$$
\frac{1}{1+e^l}=\frac{1-\sqrt{1-4/t^2}}{2}
$$
using the trace-length relation
$$
l=2\operatorname{arccosh}(t/2).
$$
See Figure \ref{randomum} for an example of the three $f$ functions for a random triple.

\begin{figure}[!hbtp]
\begin{center}
\includegraphics[scale=0.2]{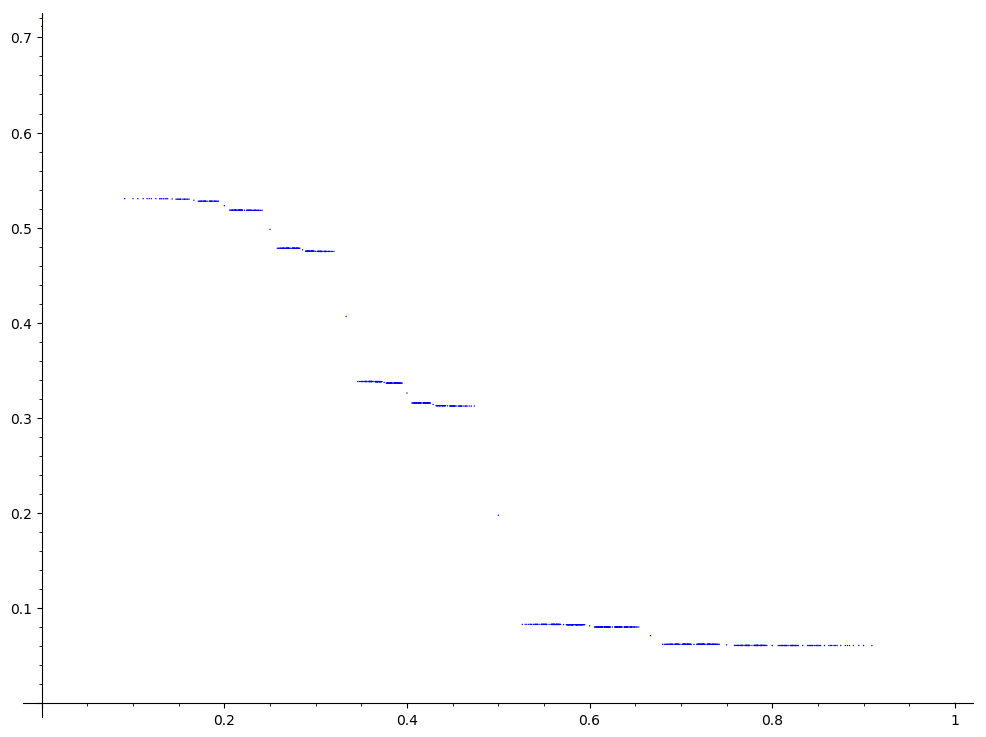}
\includegraphics[scale=0.2]{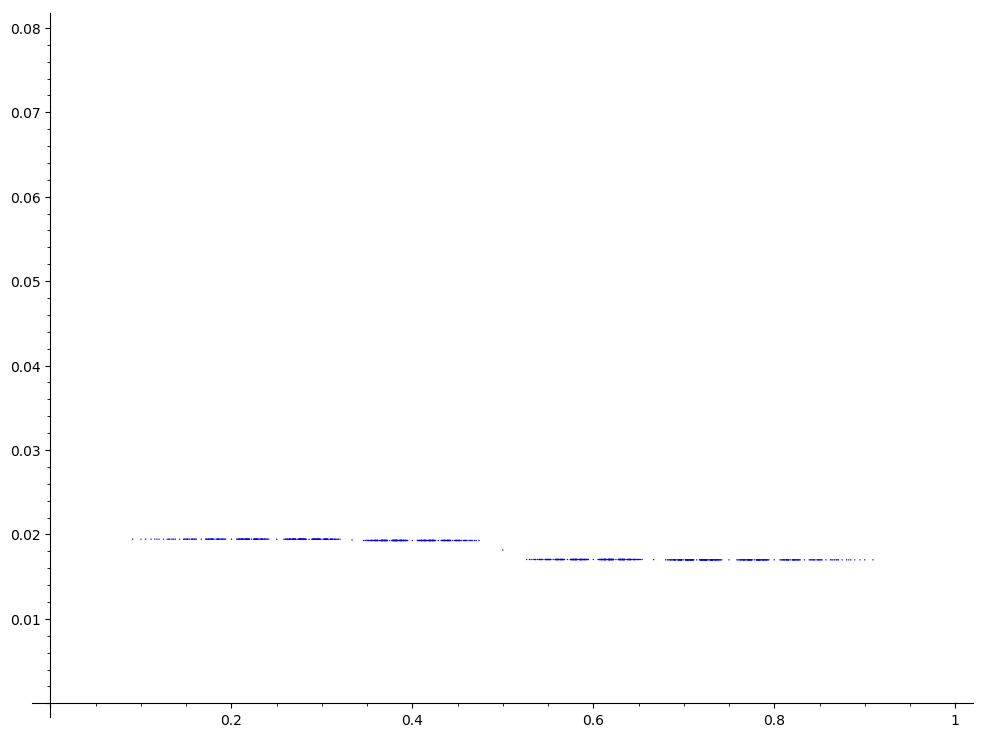}
\includegraphics[scale=0.2]{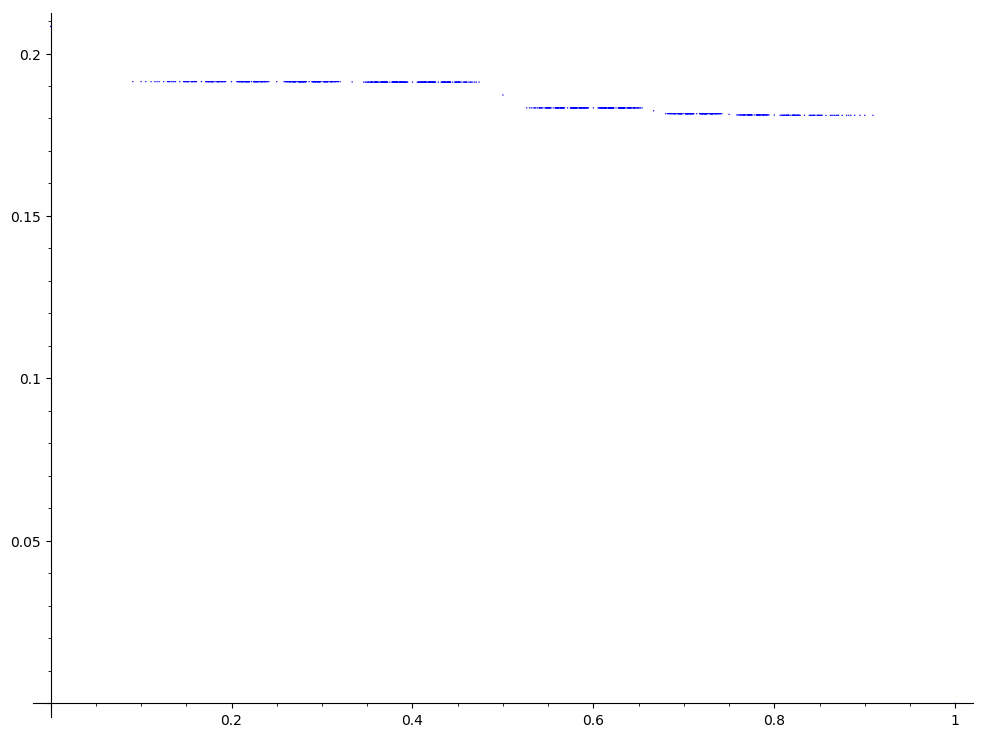}
\captionsetup{justification=raggedright,format=hang}
\caption{The $(a,b)$, $(b,c)$, and $(c,a)$ versions of $f$ for a random triple $(a,b,c)=(2.59740058623, 4.18711171215, 7.73808785195)$.}\label{randomum}
\end{center}
\end{figure}

As a final remark, the function $f$ (say with the triple $(6,3,3)$) is one-third of what is called the ``Cohn index'' in \cite{aignerbook} and has been considered elsewhere to some extent (e.g. \cite{bombieri} Appendix B).  Bounding this ratio using the existence of the jumps at $0/1$ and $1/1$ (along with a little number theory) is enough to prove the uniqueness of some Markoff numbers (prime powers and twice prime powers in \cite{bombieri}, $5p^k$ in \cite{aignerbook} Proposition 4.21).  Transcendence of the values of $f$ at irrational $x$ for the starting triple $(3,3,3)$ can be deduced from \cite{adbu} as noted in \cite{hines}.
%%%%%%%%%%%%%%%%%%%%%%%%%%%%%%%%%%%
%%%%%%%%%%%%%%%%%%%%%%%%%%%%%%%%%%%
%%%%%%%%%%%%%%%%%%%%%%%%%%%%%%%%%%%
\bibliographystyle{alpha}
\nocite{*}
\bibliography{teich_ref}
\end{document}